\documentclass[12pt]{amsart}
\usepackage[margin=1in]{geometry}
\usepackage{lmodern}
\usepackage{mathptmx}
\usepackage{mathtools}
\usepackage{amsthm}
\usepackage{amssymb}
\usepackage{amsrefs}
\usepackage{graphicx}
\usepackage{tikz}
	\usetikzlibrary{decorations.pathreplacing}
	\usetikzlibrary{patterns}
\usepackage{caption}
\usepackage{csquotes}

\MakeOuterQuote{"}
\newcommand{\ra}{\rightarrow}
\newcommand{\pr}{\prime}
\newcommand{\de}{\partial}

\newcommand{\R}{\mathbb{R}}
\newcommand{\Z}{\mathbb{Z}}

\DeclareMathOperator{\id}{id}

\newtheorem{thm}{Theorem}

\theoremstyle{definition}
\newtheorem{defin}{Definition}
\newtheorem*{defin*}{Definition}
\newtheorem{note}{Note}
\theoremstyle{plain}
\newtheorem{lemma}{Lemma}

\begin{document}

\setlength{\belowdisplayskip}{0.5em}
\setlength{\belowdisplayshortskip}{0.5em}
\setlength{\abovedisplayskip}{-0.5em}
\setlength{\abovedisplayshortskip}{-0.5em}

\title{Width of Codimension Two Knots}

\author{Michael Freedman}
\address{\hskip-\parindent
	Michael Freedman\\
	Microsoft Research, Station Q, and Department of Mathematics\\
    University of California, Santa Barbara\\
    Santa Barbara, CA 93106\\}
\email{mfreedman@math.ucsb.edu}

\author{Jonathan Hillman}
\address{\hskip-\parindent
	Jonathon Hillman\\
	School of Mathematics and Statistics\\
	University of Sydney\\
	NSW 2006, Australia\\}
\email{jonathan.hillman@sydney.edu.au}

\begin{abstract}
	We extend the classical definition of \emph{width} to higher dimensional, smooth codimension 2 knots and show in each dimension there are knots of arbitrarily large width.
\end{abstract}

\maketitle

\section{Introduction and statements of Theorems}

\begin{defin}
	Given a smooth embedding $K: S^n \hookrightarrow \R^{n+2}$, let $\pi: \R^{n+2} \ra \R^n$ be any composition $\R^{n+2}\xrightarrow{d}\R^{n+2}\ra\R^n$, 
when the first map is any diffeomorphism and the second map is projection onto the last $n$-coordinates. The \emph{width} of $K$, $w(K)$ is

\[
	w(K) = \mathrm{minimax}|K(S^n) \cup \pi^{-1}(P)|
\]
where the minimum is over the choice of product projections $\pi$ and the maximum is over regular values $p \in \R^n$ for the composition $\pi \circ K$. 
We use $|~|$ to denote cardinality.
\end{defin}

In the classical case, $n = 1$, this is nearly the usual definition of width. Two details should be noted: for the classical unknot $\mathrm{U}$ our definition gives 0 (not 2) since $\mathrm{U}$ can lie in a plane $= \pi^{-1}(p)$ and thus have no regular point. Otherwise, our definition is the same as minimizing over Morse functions the maximum intersection with a generic level. Our definitions should not be confused with a more elaborate count, \emph{Gabai width}, introduced to study properly $\R$ \cite{gabai}.

Classically, it is well known that a nontrivial $K$ has $w(K) \geq 4$ and that if $K^\pr$ is a satellite of homological degree $d$ of $K$ nontrivial, then

\[
w(K^\pr) = dw(K)
\]

Here we prove weaker analogs of these classical facts for $n>1$. To formulate these we define \emph{homological width} $w_H(K)$, and for this we set $\overline{K}:S^n \times D^2 \hookrightarrow \R^{n+2}$ to be a real analytic embedding (for a technical reason\footnote{All smooth maps and structures may be approximated by real-analytic version. We need this version at some point to stratify and triangulate non-regular values.}) of the tubular neighborhood of $K$, and $d:\R^{n+2} \ra \R^{n+2}$ also now real analytic.

\begin{defin}
	$w_H(K) = \mathrm{minimax}[\overline{K}(S^n \times D^2 \cup \pi^{-1}(p)]$, where again the minimum is over product projections $\pi$ and the maximum is over regular values $p \in \R^n$. But now the square brackets $[-]$ denotes "the number of connected components of -- which represent a non-zero homology class in $H_2(S^n \times D^2, \de;\Z)$".
\end{defin}

\begin{note}
	An easy homological argument shows $w(K) \geq w_H(K)$
\end{note}

\begin{note}
	In the classical case, $n = 1$, $w_H(\text{unknot}) = 0$; the unknotted solid torus in $\R^3$ may be sliced by homologically trivial annuli. This explains why width is only multiplicative under satellite for \emph{non-trivial} knots.
\end{note}

\begin{note}
	For notational convenience we have located our knots in $\R^{n+2}$ rather than its one-point-compactification $S^{n+2}$. This makes it slightly easier to discuss projection onto $\R^n$.
\end{note}

\begin{defin}
	We say a knot $K: S^n \hookrightarrow \R^{n+2}$ is \emph{cohomologically rich} if the knot complement 
$Y \coloneqq \R^{n+2} \backslash K(S^n)$ admits a covering space $\widetilde{Y}$ with a nontrivial $(n+1)$-fold cup product. 
That is, for some coefficient field $F$ there are classes $\alpha_0, \dots, \alpha_n \in H_1(\widetilde{Y}; F)$ so that 
$\alpha_0 \cup \dots \cup \alpha_n \neq 0 \in H^{n+1}(\widetilde{Y}; F)$. 
If $K$ is not cohomologically rich we call it \emph{cohomologically poor}. In our examples it suffices to set $F = \Z_2$.
\end{defin}

\begin{note}
	All nontrivial classical knots $S^1 \hookrightarrow S^3$, are cohomologically rich. 
	It is sufficient to consider the cover induced by the inclusion of the peripheral torus.
	On the other hand, if an $n$-knot is cohomologically rich then its knot group has cohomological dimension $\geq{n+1}$.
	In particular, Artin (untwisted) spins of classical knots (and their higher dimensional analogs) are cohomologically poor,
	since classical  knot groups have cohomological dimension $\leq2$.
         We shall use {\it twist}-spins (in \S3) to construct examples of cohomologically rich $n$-knots.
\end{note}

\begin{thm}
	If $w_H(K) = 0$ then $K$ is cohomologically poor.
\end{thm}

\begin{thm}
	If $K^\pr$ is a satellite of $K$ of homological degree $d_H$ then $w(K^\pr) \geq d_H w_H(K)$.
\end{thm}

To see that satellites of homological degree $d > 1$ indeed exist in all dimensions, 
consider the $(d,1)$-torus knot $C$ on $S^1 \times S^1$ pushed into interior $(S^1 \times D^2)$. 
This is the ur-example in dimension $n = 1$, 
and $(n-1)$-fold suspension produces examples in all dimension of $S^n \subset S^n \times D^2$ of degree $= d$.

\begin{thm}
	For every $n \geq 1$ there exist smooth $K: S^n \hookrightarrow \R^{n+2}$ which are cohomologically rich.
\end{thm}

And immediate consequence of Theorems 1, 2, and 3 is

\begin{thm}
	For all $n \geq 1$ there are smooth knots $K: S^n \hookrightarrow \R^{n+2}$ with arbitrarily large width $w(K)$.
\end{thm}

\begin{proof}[Proof of Theorem 4]
	It is only necessary to have a seed $K$, provided by Theorem 3, with $w_H(K) > 0$ in each dimension. 
	Such a $K$ has positive $w_H(K)$ so a satellite $K^\pr$ will, by Theorem 2, have $w(K^\pr) \geq dw_H(K) \geq d$.
\end{proof}

\begin{proof}[Proof of Theorem 2]
	Since width is defined using regular values we may assume any intersection with the product 2-planes of interest, $\pi^{-1}(p)$, $p$ regular, and $K^\pr$ must have $d|d_j|$ transverse intersections with the $j^{\text{th}}$ component $Q_j$ of $\overline{K}(S^n \times D^2 \cup \pi^{-1}(p))$, where $[Q_j] = d_j \in H_2(S^n \times D^2, \de; \Z)$. By the definition of $w_H(K)$ there will be at a maximal 2-plane with precisely $w_H(K)$ indices $j$ for which $|d_j| \geq 1$, so $w(K^\pr) \geq \sum_j d|d_j| \geq d$.
\end{proof}

\section{Proof of Theorem 1}

Let us use the notation $X = \overline{K}(S^n \times D^2)$ and $Y$ the open complement, $Y \coloneqq \R^{n+2} \backslash X$. We need to \emph{assemble} $X$ and even more importantly $Y$ from the preimages $X_p = \pi^{-1}(p)$ and $Y_p = \R^2_p \backslash X_p$. To do this we use analyticity to control the singularities of the composition

\begin{equation}
	S^n \times D^2 \xrightarrow{\overline{K}} \R^{n+2} \xrightarrow{\pi} \R^n
\end{equation}

First we use a two dimensional analysis to understand each $Y_p$ in light of the hypothesis $w_H(K) = 0$ which allows us to assume every component $Q$ of every regular $X_p$ is inessential, $[Q] = 0 \in H_2(S^n \times D^2, \de; \Z)$. The following lemma shows that for generic $p$, $Y_p \hookrightarrow Y$ is null homotopic.

\begin{lemma}\label{scc}
	Assume $n>1$ and that we are in the situation where for all generic $p \in \R^n$ all components $Q$ of $X_p$ represent the trivial relative class in $H_2(S^n \times D^2, \de; \Z)$. Let $\gamma$ be any scc parallel to an end of $Y_p$, then the degree of that end $[\gamma] = 0 \in H_1(S^n \times S^1;\Z)$, and therefore the inclusion $Y_p \in Y$ is null homotopic.
\end{lemma}

\begin{proof}
	Let $\Delta$ be the disk bounded by $\gamma$ in $\R^2_p = \pi^{-1}(p)$. $\Delta$ is alternately colored white and black by the parts outside (inside $\overline{K}(S^n \times D^2)$). The frontier circles each carry an integral degree in $H_1(S^n \times S^1;\Z)$. Quite generally these degrees must add to zero over the boundary of a white component (total linking number with the knot vanishes over the boundary of a necessarily null homologous white cycle), and our homological hypothesis implies that the same holds for black components as well. Thus we see a tree of sccs whose leaves (innermost circles) all have degree zero. The conservation laws just described implies that its root $\gamma$ also has degree zero.

	For planer surfaces such as $Y_p$ loops parallel to the ends normally generate the fundamental group. Since these map homologically trivially into $S^n \times S^1$, they are moreover null homotopic there. It follows that $\pi_1(Y_p) \ra \pi_1(Y)$ is zero and that $Y_p$ is null homotopic in $Y$.
\end{proof}

The singularities of $\pi \circ K$ are governed by Lojasiewicz's Lemma \cite{l64}.

\begin{lemma}\label{subc}
	The irregular values of $\pi \circ K$ can be triangulated as a finite $(n-1)$-subcomplex $J$ of $\R^n$
\end{lemma}

It is worth picturing what happens near $J$. Off $J$ $X_p$ is a compact smooth 2-manifold in $\R^2_p$. As $p$ moves toward $J$, bits of $X_p$ will pinch or join or birth/deaths appear. Across the $n-1$ cells $J_{n-1}$, of $J$, these are the familiar Morse singularities. Along $n-j$ cells various "codimension $j$" singularities occur which include, but are not limited to $j$ disjoint applications of Morse moves.

\begin{lemma}
	For all $p \in \R^n$, regular or irregular, the inclusion $Y_p \subset Y$ is null homotopic.
\end{lemma}

\begin{proof}
	In Lemma \ref{scc} we argued this for generic $p$ be showing each end-parallel $\gamma$ was null homotopic in $Y$. But for a nongeneric $p$ each end parallel $\gamma$ is \emph{also} end parallel in some nearby $Y_{p^\pr}$, $p^\pr$ generic, so the earlier argument still applies.
\end{proof}

Now take a fine handle decomposition of $\R^n$, fine with respect to the simplicities of $J$ so that the preimage of each $i$-handle $h$, $0 \leq i \leq n$, is the versal unfolding of some singularity of codimension $i$.

\begin{lemma}
	For each such handle $h$, let $Y_h \coloneqq Y \cup \pi^{-1}(h)$ and $Y_h^c \subset Y_h$ be an arbitrary connected component. All the inclusions $Y_h^c \subset Y$ induce the trivial map $\pi_1(Y^c_h) \xrightarrow{0} \pi_1(Y)$.
\end{lemma}

\begin{proof}
	We need to show how to take a loop $\gamma \subset Y_h^c$ and contract it in $Y$. If $\gamma$ projects to a point $p \in h$ by Lemma 3 we may perturb $\gamma$ so $p$ is generic and apply Lemma \ref{scc}.

	Again perturbing $\gamma$ we may cross the Morse strata $J_{n-1}$ in $h \in \R^n$ transversely and a null homotopy $\delta$ for $\mathrm{image}(\gamma) \subset h$ meets the codimension 2 strata transversely. By subdividing $\delta$ is sufficient for us to check that the lift $\gamma_0$ of any small loop $\overline{\gamma}_0 \coloneqq \partial \delta_0$ encircling a codimension two strata $J_{n-2}$ is null homotopic in $Y$. It is well known (e.g. \cite{SR}) that $J_{n-2}$ consists of disjoint Morse singularities (saddles and/or birth/death) and cubic cusps. For reference, we illustrate several representative cases, the last two being most interesting.

	\begin{figure}[!ht]
		\begin{tikzpicture}
			\draw (-2, 0) -- (2, 0);
			\draw (0, -2) -- (0, 2);
			\draw (-1, 1) circle (0.25);
			\draw (-1.3, 0.7) -- (-0.7, 1.3);
			\draw (1, 1) circle (0.25);
			\draw (-1, -1) circle (0.25);
			\draw (1, -0.67) circle (0.25);
			\draw (1, -1.33) circle (0.25);
			\node at (0,-2.3) {($\xrightarrow{\text{birth}}$, birth$\downarrow$)};
			\node at (0, -2.9) {(a)};

			\draw (4, 0) -- (8, 0);
			\draw (6, -2) -- (6, 2);
			\draw (5, 1) circle (0.5);
			\draw (5, -1) circle (0.5);
			\draw (5, -1) circle (0.1);
			\draw (6.5,0.7) arc (-90:-270:0.3);
			\draw (6.5, 1.3) to [out = 0, in = 90] (7, 1) to [out = -90, in = 0] (6.5, 0.7);
			\draw (7.7, 0.7) to [out = 180, in = -90] (7.2, 1) to [out = 90, in = 180] (7.7, 1.3);
			\draw (7.7, 1.3) arc (90:-90:0.3);
			\draw (6.5, -1.3) arc (-90:-270:0.3);
			\draw (6.5, -0.7) to [out = 0, in = 90] (7, -1) to [out = -90, in = 0] (6.5, -1.3);
			\draw (7.7, -1.3) to [out = 180, in = -90] (7.2, -1) to [out = 90, in = 180] (7.7, -0.7);
			\draw (7.7, -0.7) arc (90:-90:0.3);
			\draw (6.5, -1) circle (0.1);
			\node at (6,-2.3) {($\xrightarrow{\text{saddle}}$, birth$\downarrow$)};
			\node at (6, -2.9) {(b)};

			\draw (10, 0) -- (14, 0);
			\draw (12, -2) -- (12, 2);
			\draw (11, 1) circle (0.5);
			\draw (12.5,0.7) arc (-90:-270:0.3);
			\draw (12.5, 1.3) to [out = 0, in = 90] (13, 1) to [out = -90, in = 0] (12.5, 0.7);
			\draw (13.7, 0.7) to [out = 180, in = -90] (13.2, 1) to [out = 90, in = 180] (13.7, 1.3);
			\draw (13.7, 1.3) arc (90:-90:0.3);
			\draw (11.3, -0.5) arc (0:180:0.3);
			\draw (11.3, -0.5) to [out=-90, in=0] (11,-1) to [out=180, in=-90] (10.7, -0.5);
			\draw (11, -1.2) to [out = 0, in = 90] (11.3, -1.7);
			\draw (11, -1.2) to [out = 180, in = 90] (10.7, -1.7);
			\draw (11.3, -1.7) arc (0:-180:0.3);
			\draw (12.2,-0.5) arc (180:90:0.3);
			\draw (12.2,-0.5) to [out = -90, in = 180] (12.5,-1);
			\draw (12.5,-0.2) to [out = 0, in = 90] (13,-0.5);
			\draw (12.5,-1) to [out = 0, in = -90] (12.7,-0.8) to [out = 90, in = 180] (12.85,-0.65) to [out = 0, in = -90] (13,-0.5);
			\draw (12.5,-1.2) to [out = 0, in = 90] (12.8,-1.7);
			\draw (12.5,-1.2) to [out = 180, in = 90] (12.2,-1.7);
			\draw (12.8,-1.7) arc (360:180:0.3);
			\draw (12.2,-0.5) arc (180:90:0.3);
			\draw (13.7,-0.8) to [out = 180, in = -90] (13.2,-0.5) to [out = 90, in = 180] (13.7,-0.2);
			\draw (13.7,-0.2) arc (90:-90:0.3);
			\node at (12,-2.3) {($\xrightarrow{\text{saddle}}$, saddle$\downarrow$)};
			\node at (12, -2.9) {(c)};

			\draw (1, -6) -- (5, -6);
			\draw (3, -4) -- (3, -8);
			\draw (2, -7) ellipse (0.7 and 0.3);
			\draw (2, -7) ellipse (0.9 and 0.5);
			\draw (2.9, -5) arc (0:150:0.9 and 0.5);
			\draw (2.9, -5) arc (0:-150:0.9 and 0.5);
			\draw (2.7, -5) arc (0:150:0.7 and 0.3);
			\draw (2.7, -5) arc (0:-150:0.7 and 0.3);
			\draw (1.4, -5.16) arc (35:205:0.1);
			\draw (1.4, -4.84) arc (-35:-205:0.1);
			\draw (3.1,-7) arc (180:30:0.9 and 0.5);
			\draw (3.1,-7) arc (-180:-30:0.9 and 0.5);
			\draw (3.3,-7) arc (180:30:0.7 and 0.3);
			\draw (3.3,-7) arc (-180:-30:0.7 and 0.3);
			\draw (4.6,-7.16) arc (180-35:180-205:0.1);
			\draw (4.6,-6.84) arc (180+35:180+205:0.1);
			\draw (4.6,-5.16) arc (180-35:180-205:0.1);
			\draw (4.6,-4.84) arc (180+35:180+205:0.1);
			\draw (3.4, -5.16) arc (35:205:0.1);
			\draw (3.4, -4.84) arc (-35:-205:0.1);
			\draw (3.4, -5.16) [out = -35, in = 215] to (4.6, -5.16);
			\draw (3.23, -5.26) [out = -35, in = 215] to (4.77, -5.26);
			\draw (3.4, -4.84) [out = 35, in = -215] to (4.6, -4.84);
			\draw (3.23, -4.74) [out = 35, in = -215] to (4.77, -4.74);
			\node at (3,-8.3) {($\xrightarrow{\text{saddle}}$, saddle$\downarrow$)};
			\node at (3,-8.9) {(d)};

			\draw (9, -4) arc (0:-90:2);
			\draw (9,-8) arc (0:90:2);
			\draw  (7.3,-7.3) circle (0.1);
			\draw  (8,-7.3) circle (0.4);

			\draw (7,-5.1) arc (-90:-270:0.3);
			\draw (7,-4.5) to [out = 0, in = 90] (7.5,-4.8) to [out = -90, in = 0] (7,-5.1);
			\draw (8.2,-5.1) to [out = 180, in = -90] (7.7,-4.8) to [out = 90, in = 180] (8.2,-4.5);
			\draw (8.2,-4.5) arc (90:-90:0.3);
			\draw (10.1,-5.5) arc (90:-90:0.5);
			\draw (9.1, -6) to [out = 90, in = 180] (10.1,-5.5);
			\draw (9.1, -6) to [out = -90, in = 180] (10.1,-6.5);
			\node at (9,-8.3) {(cusp or "degenerate neck pinch")};
			\node at (9,-8.9) {(e)};
		\end{tikzpicture}
		\caption{}
	\end{figure}

	We need to see that we can lift the obvious small null homotopy of $\overline{\gamma_0}$ in $\delta_0$ to a null homotopy of $\gamma_0$ in the preimage $Y_{\delta_0} \subset Y_h$. The difficulty is that $Y_{\delta_0} \ra \delta_0$ is not a fibration but the combinations is easily checked by hand.

	In cases (a), (b), (c), and (e) of Figure 1 one deforms $\gamma_0$ into the open quadrant $k$ with most circles---in fact, $Y_{\delta_0}$ deformation retracts to the preimage $Y_k$ of $k \subset \delta_0$. Case (d) deformation retracts along the positive diagonal to the 3D model sketched in Figure 2. In this model we readily check that any loop in the surface complement is homotopic to an end in either the upper or lower generic (compactified) planar slice, and thus by Lemma 1 maps trivially to $\pi_1(Y)$.

	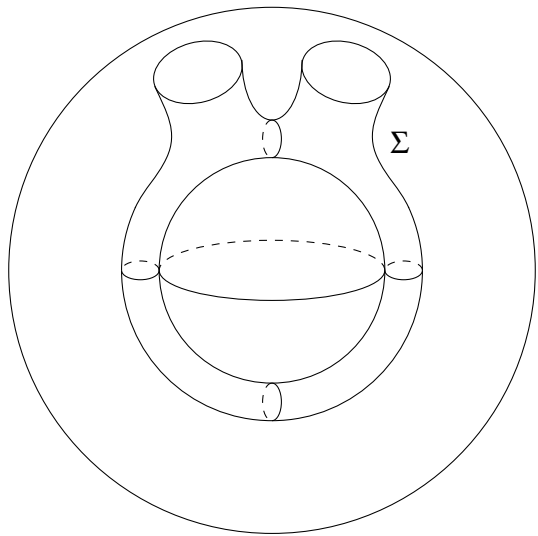
\begin{figure}[!ht]
		\begin{tikzpicture}
			\draw (0,0) circle (3.5);
			\draw (0,0) circle (1.5);
			\draw (1.5,0) arc (0:-180:1.5 and 0.4);
			\draw [dashed] (1.5,0) arc (0:180:1.5 and 0.4);
			\draw (2,0) arc (0:-180:2);
			\draw (2,0) arc (0:25:2);
			\draw (-2,0) arc (180:155:2);

			\draw [rotate=15] (-0.27,2.8) ellipse (0.6 and 0.4);
			\draw [rotate=-15] (0.27,2.8) ellipse (0.6 and 0.4);
			\draw (-0.4,2.8) arc (-180:0:0.4 and 0.8);
			\draw (-1.82, 0.83) to [out = 65, in = 245] (-1.4, 1.5) to [out = 65, in = -65] (-1.56, 2.45);
			\draw (1.82, 0.83) to [out = 115, in = -65] (1.4, 1.5) to [out = 115, in = 245] (1.56, 2.45);
			\draw (0, 2) arc (90:-90:0.125 and 0.25);
			\draw [dashed] (0, 2) arc (90:270:0.125 and 0.25);
			\draw (0, -1.5) arc (90:-90:0.125 and 0.25);
			\draw [dashed] (0, -1.5) arc (90:270:0.125 and 0.25);
			\draw (2,0) arc (0:-180:0.25 and 0.125);
			\draw [dashed] (2,0) arc (0:180:0.25 and 0.125);
			\draw (-1.5,0) arc (0:-180:0.25 and 0.125);
			\draw [dashed] (-1.5,0) arc (0:180:0.25 and 0.125);
			\node at (1.7,1.7) {$\Sigma$};
		\end{tikzpicture}
		\caption{Spheres are (compactified) planar slices. The surface $\Sigma$ is the trace of Figure 1(d) along the major diagonal.}
	\end{figure}

	This completes the proof of Lemma 4.
\end{proof}

\newpage
We are ready to prove Theorem 1. Consider the diagram:

\begin{tikzpicture}
	\node at (0,0) {$\widetilde{Y}$};
	\node at (0.5, 0) {$\rightarrow$};
	\node at (1, 0) {$Y$};
	\node at (1.5,0) {$\hookrightarrow$};
	\node at (2.3,0) {$\mathbb{R}^{n+2}$};
	\node at (3,0) {$\rightarrow$};
	\node at (3.5,0) {$\mathbb{R}^n$};

	\node [rotate=90] at (0,-0.5) {$\subset$};
	\node [rotate=90] at (1,-0.5) {$\subset$};
	\node [rotate=90] at (3.5,-0.5) {$\subset$};
	\node at (0,-1) {$\widetilde{Y}_h$};
	\node at (1,-1) {$Y_h$};
	\node at (0.5,-1) {$\rightarrow$};
	\node at (3.5,-1) {$h$};
	\node at (2.3, -1) {$\xrightarrow{\hspace{3em}}$};
\end{tikzpicture}

Each component of $Y_n$ maps $\pi_{1}$-trivially into $Y$, so in the covering space each component of the preimage $\widetilde{Y}_h$ maps trivially into $\widetilde{Y}$. Consider the handle index $i$ for each $i$ a "color." The $\{\widetilde{Y}_h\}$ divide $\widetilde{Y}$ into (possibly disconnected but non-pathological) "tiles" of $n+1$ colors, $0 \leq i \leq n$, and each tile's inclusion induces the zero map

\[
H_1(\widetilde{Y}_h;F) \ra H_1(\widetilde{Y};F) \tag{$\ast$}
\]

\noindent
on homology (with coefficients in an arbitrary field $F$). We may now make a "Lusternik-Schnirelmann" style argument. First group the $Y_{h_i}$ into larger groups $Y_i$.

Let $\widetilde{Y}_i = \bigcup_{\text{all handles of index }i} \widetilde{Y}_{h_i}$, $\widetilde{Y} = \bigcup_{i=0}^n \widetilde{Y}_i$. Let $\alpha_0, \dots, \alpha_n$ be $n+1$ elements of $H^1(\widetilde{Y};F)$. Consider the exact sequence:

\[
H^1(\widetilde{Y}, \widetilde{Y}_i;F) \xrightarrow{\text{inc}^\ast} H^1(\widetilde{Y};F) \xrightarrow{\text{inc}^\ast} H^1(\widetilde{Y}_i;F) \xrightarrow{\delta} H^2(\widetilde{Y}, \widetilde{Y}_i; F)
\]

$\delta$ is an injection since dually the boundary map $H_2(\widetilde{Y}, \widetilde{Y}_i, F) \xrightarrow{\de} H_1(\widetilde{Y}_i;F) \ra 0$ is onto since by ($\ast$) the map to the next term is zero. Thus the first $\text{inc}^\ast$ is a surjection, so lift $\alpha_i$ to $\overline{\alpha}_i \in H^1(\widetilde{Y}, \widetilde{Y}_i; F)$, $0 \leq i \leq n$.

$\overline{\alpha}_0 \cup \cdots \cup \overline{\alpha}_n \in H^{n+1}(\widetilde{Y}, \bigcup \widetilde{Y}_i; F) = H^{n+1}(\widetilde{Y}, \widetilde{Y}; F) \cong 0$, but any cochain representative for $\overline{\alpha}_i$ is also a cochain representative for $\alpha_i$, $0 \leq i \leq n$, so the cup product computation may be done with representatives $\{\overline{\alpha}_i, 0 \leq i \leq n\}$. Thus $\alpha_0 \cup \cdots \cup \alpha_n = 0 \in H^n(\widetilde{Y};F)$, showing that $K$ is cohomologically poor.
\qed

\section{Examples}

Since all classical knots are cohomologically rich we begin our induction with the famous Cappell-Shaneson 2-knot $K_2$ 
with fiber a punctured 3-torus $T^3_-$ 
and linear monodromy matrix 

\[
A_2=\begin{bmatrix} 0 & 1 & 0 \\ 0 & 0 & 1 \\ -1 & 1 & 0 \end{bmatrix}.
\] 
See \cite{cappells}. Initially there was a question if this fibered knot lay in $S^4$ or perhaps an exotic homotopy 4-sphere, 
however, the question was resolved in favor of $S^4$ by Gompf in \cite{gompf}. 
Picking $\Z_2$ as our coefficient field, consider the image of $A_2$ in the finite group $SL(3, \Z_2)$, 
where it has some finite order $t_2$.

Let $\widetilde{Y}$ be the $t_2$ cyclic cover of $Y \coloneqq S^4 \setminus K_2$. We claim $\widetilde{Y}$ has the same $Z_2$-cohomology and 
cup product structure as $T^3_- \times S^1$. The closed manifold $M$ in Lemma 5 is built from the same monodromy $t_2A_2$ as $\widetilde{Y}$ but applied to the closed $T^3$ instead of $T^3_-$. $\widetilde{Y}$ is obtained from $M$ by deleting a circle representing the class $U$; the products $XYU$, $XZU$, and $YZU$ remain nontrivial, showing that $Y$ is cohomologically rich.

According to Zeeman \cite{zeeman65}, the $t_2$-twist-spun knot $K_3: S^3 \hookrightarrow S^5$ will also be fibered  
with fiber $F^4_-$ the (punctured) $t_2$-fold cyclic branched cover of $S^4$ around $K_2$. 
The closed fiber $F^4$ will be a $\Z_2$-cohomology torus: $A_2^{t_2} = \id \in SL(3, \Z_2)$, so $F^4\backslash$meridional circle has the
$\Z_2$-cohomology of  $T^3_- \times S^1$. 

Above the classical dimension there is only one possibility for regluing the meridional circle so $F^4$ is a $\Z_2$-cohomology $T^4$, 
in the sense of admitting a map from $T^4$ inducing an isomorphism on $H^\ast(-;\Z_2)$.
That it actually has the correct ring structure follows from the following lemma.

\begin{lemma}
Let $M$ be the mapping torus of the diffeomorphism of 
$T^3=\R^3/\Z^3$ determined by a matrix $A\in{SL(3,\Z)}$.
If  $A\equiv{I}$ mod $(p)$ for some prime $p$ then
$H^*(M;\Z_p)$ is generated by $H^1(M;\Z_p)$.
\end{lemma}

\begin{proof}
Let $x,y,z\in\pi=\pi_1(M)$ represent a basis for the image of $\pi_1(T^3)=\Z^3$,
and let $u\in\pi$ generate the quotient.
Then $H_1(M;\mathbb{Z}_p)\cong\Z_p^4$.
Let $U,X,Y,Z$ be the Kronecker dual basis for $H^1(M;\Z_p)$
(so $U(u)=1$ and $U(x)=U(y)=U(z)=0$, etc.).
Then $XYZ\not=0$ in $H^3(M;\Z_p)$, 
since it has nonzero restriction to $H^3(T^3;\Z_p)$.
Hence $UXYZ\not=0$ in $H^4(M;\Z_p)$,
by the non-degeneracy of Poincar\'e duality.
Hence each of the triple products of these basis elements is non-zero.
As they are clearly the Poincar\'e duals of the basis for $H_1(M;\Z_p)$
represented by $u,x,y,z$, they are a basis for $H^3(M;\Z_p)$, and
the lemma follows easily.
\end{proof}

Now inductively apply \cite{zeeman65}. 
Given $K^n: S^n \hookrightarrow S^{n+2}$, fibered with fiber a punctured $\Z_2$-cohomology $T^{n+2}$, 
let $t_n$ be the order of the monodromy measured in $SL(n+1, \Z_2)$. 
The $t_n$-twist-spin of $K^n$, $K^{n+1}: S^{n+1} \hookrightarrow S^{n+3}$ is again fibered with fibered with fiber the punctured, 
$t_n$-fold cyclic branched cover of $S^{n+2}$ about $K^n(S^n)$. 
Again, since the monodromy of the $t_n^{\text{th}}$ power is trivial in $SL(n+1;\Z_2)$ the closed fiber $F^{n+2}$
is a $\Z_2$-cohomology $T^{n+2}$. 
It in turn has a monodromy of finite order $t_{n+1}$ in $SL(n+2, \Z_2)$ and thus a $t_{n+1}$-fold cyclic cover $\widetilde{Y}^{n+1}$ 
which is diffeomorphic to $F^{n+2}_- \widetilde{\times} S^1$, with 
the $\Z_2$ cohomology ring structure of $T^{n+2}_- \times S^1$, 
and nontrivial $n+1$-fold cup products.

\begin{note}
The matrix $A_2$ actually has order $t_2=7$, 
since the {\it mod-}2 reduction of the characteristic polynomial 
$\det(tI-A_2)=t^3+t+1$ divides $t^7-1$ in $\Z_2[t]$.
In fact every $3\times3$ Cappell-Shaneson matrix has image of order 7 in $SL(3,\Z_2)$,
since the only other possibility for the {\it mod-}2 reduction of the characteristic polynomial  
is  $t^3+t^2+1$, which also divides $t^7-1$ in $\Z_2[t]$.
There is no such universal value for $t_n$ when $n>2$.
\end{note}

\end{document}